\documentclass{amsart}

\usepackage{amsmath} 
\usepackage{amsfonts} 
\usepackage{amssymb} 
\usepackage{epsfig} 
\usepackage{graphicx} 
\usepackage{multirow} 
\usepackage{verbatim} 
\usepackage{rotating} 
\usepackage[all]{xy} 

\newtheorem{theorem}{Theorem}[section]
\newtheorem{lemma}[theorem]{Lemma}

\theoremstyle{definition}
\newtheorem{definition}[theorem]{Definition}
\newtheorem{example}[theorem]{Example}

\newtheorem{remark}[theorem]{Remark}
\newtheorem{notation}[theorem]{Notation}
\numberwithin{equation}{section}
\newtheorem{corollary}[theorem]{Corollary}

\newcommand{\Map}{\mathbf{Map}}
\newcommand{\map}{\mathbf{map}}

\newcommand{\HOM}{\mathbf{HOM}}
\newcommand{\homs}{\mathbf{hom}}
\newcommand{\sSet}{\mathbf{sSet}}

\newcommand{\Cat}{\mathbf{Cat}}

\newcommand{\Top}{\mathbf{Top}}

\newcommand{\C}{\mathbf{C}}

\newcommand{\D}{\mathbf{D}}
\newcommand{\A}{\mathbf{A}}
\newcommand{\B}{\mathbf{B}}
\newcommand{\M}{\mathbf{M}}
\newcommand{\sCat}{\mathbf{sCat}}

\newcommand{\K}{\mathcal{K}}

\newcommand{\V}{\mathbf{V}}

\newcommand{\N}{\mathrm{N}_{\bullet}}
\newcommand{\iso}{\mathrm{iso}}
\newcommand{\sing}{\mathrm{sing}}
\newcommand{\Ho}{\mathrm{Ho}}

\newcommand{\colim}{\mathrm{colim}}
\newcommand{\Ob}{\mathrm{Ob}}

\begin{document}

\title{A Model Structure on $\Cat_{\Top}$}
\author{Amrani Ilias}
\address{Department of Mathematics, Masaryk University, Czech republic}
\email{ilias.amranifedotov@gmail.com}


\subjclass[2000]{Primary 55, 18.}

\date{April 30, 2011.}



\begin{abstract}
In this article, we construct a cofibrantly generated Quillen model structure on the category of small
topological categories $\Cat_{\Top}$. It is Quillen equivalent to the Joyal  model structure
of $(\infty,1)$-categories and the Bergner model structure on $\Cat_{\sSet}$.    
\end{abstract}

\maketitle

\section*{Introduction}

In the section \ref{section1}, we construct a Quillen model structure on the category of small 
topological categories  $\Cat_{\Top}$ \cite{Amrani}. The main advantage is the fact that all 
objects in  $\Cat_{\Top}$ are fibrant.
We show that this model structre is Quillen equivalent to the model structure on the category
 of small simplicial categories $\Cat_{\sSet}$ defined in \cite{bergner}. \\
 Why we are interested on topological categories? In \cite{Hovey}, it is shown that any model 
 category $\M$ is naturally enriched over $\sSet$ or $\Top$. The enrichment give us a higher 
  homotopical information about $\M$. \\
  
  In the topological setting, the cohomolgy theories are defined directly from the mapping space
  in the model category of topological spectra. \\
  Our future goal is to define algebraic $\K$-theory \cite{Amrani1} for a larger class of categories.


\section{ Category of small topological categories.}\label{section1}
In this article, the category of weakly Hausdorf compactly generated topological spaces
will be denoted by $\Top$ which is simplicial monoidal model category.\\
Before to start the main theorem of this section we will introduce some notations and definitions.\\

A topological category is a category enriched over $\Top$. The Category of all (small) topological 
categories is denoted by  $\Cat_{\Top}$. The morphisms in $\Cat_{\Top}$ are the enriched functors.
It is complete and cocomplete category.
 
\begin{theorem}\label{top-cat}\cite{Amrani}
 The category $\Cat_{\Top}$ admit a cofibranty generated model structure defined as follow.\\
The weak equivalences $\mathrm{F}:\C\rightarrow\D$ satisfy the following conditions.
\begin{enumerate}
\item [WT1 : ] The morphsim $\Map_{\C}(a,b)\rightarrow\Map_{\D}(\mathrm{F}a,\mathrm{F}b)$
 is a weak equivalence in the category  $\Top.$
\item [WT2 : ] The induiced morphism $\pi_{0}\mathrm{F}:\pi_{0}\C\rightarrow\pi_{0}\D$
 is a categorical equivalence in $\Cat$. 
\end{enumerate}
The fibrations are the morphisms  $\mathrm{F}:\C\rightarrow\D$ which satisfy :
\begin{enumerate}

\item [FT1 : ] The morphism $\Map_{\C}(a,b)\rightarrow \Map_{\D}(\mathrm{F}a,\mathrm{F}b)$
 is a fibration in $\Top.$
\item [FT2 : ] For each objects $a$ and $b$ in $\C$, and a weak equivalence of homotopy\\
$e:~\mathrm{F}(a)\rightarrow b$ in $\D$, there exists an object 
 $a_{1}$ in $\C$ and a weak homotopy equivalence  $d:~a\rightarrow a_{1}$ in $\C$
such that $\mathrm{F}d=e.$ 
\end{enumerate}
More over, the set $I$ of generating cofibrations is given by :
\begin{enumerate} 
\item [CT1 : ]  $|U\partial\Delta^{n}|\rightarrow |U\Delta^{n}|$, for $n\geq 0.$
\item [CT1 : ] $\emptyset \rightarrow \{x\}$, where $\emptyset $ is the empty topological category and 
$\{x\}$ is the category with one object and one morphism.
\end{enumerate}
The set $J$ of generating acyclic cofibrations is given by:
\begin{enumerate}
\item [ACT1 : ]  $|U\Lambda^{n}_{i}|\rightarrow |U\Delta^{n}|$, for $0\leq n$ and  $0\leq i \leq n.$
\item [ACT2 : ]  $\{x\}\rightarrow |\mathcal{H}|$ where $\{\mathcal{H}\}$ as defined in \cite{bergner}.
\end{enumerate}
\end{theorem}

\begin{remark}
All objects in $\Cat_{\Top}$ are \textbf{fibrant}. 
\end{remark}

\section{Proof of the main theorem}
We start by a useful lemma which gives us conditions to transfer a model structure by adjunction.
  \begin{lemma}\label{lem1}[\cite{worytkiewicz2007model}, proposition 3.4.1]
Let an adjunction $$\xymatrix{
\M \ar@<1ex>[r]^-{G} & \C\ar@<1ex>[l]^-{F}
}$$
where $\M$ is cofibrantly generated model category, with $\mathrm{I}$ generating cofibrations and $\mathrm{J}$
generating trivial cofibrations. We pose 
\begin{enumerate}
\item $\mathrm{W}$ The class of morphisms in $\C$ such the image by  $F$ is a weak equivalence in $\M$.
\item $\mathrm{F}$ The class of morphisms in $\C$ such the image by  $F$ is a fibration in $\M$.
\end{enumerate}
We suppose that the following conditions are verified:
\begin{enumerate}
\item The domain of $G (i)$ are small with respect to  $G (\mathrm{I} )$ for all $i \in \mathrm{I} $ and the domains of  $G(j)$ are small with respect to
$G(\mathrm{J} )$ for all $ j \in \mathrm{J}$.
\item The functor $F$ commutes commutes with directed colimits i.e.,  
$$F\colim (\lambda \rightarrow \C)= \colim F(\lambda\rightarrow \C).$$
\item Every transfinite composition of weak equivalences in  $\M$ is a weak equivalence.
\item The pushout of $G(j)$ by any morphism $f$ in  $\C$ is in  $\mathrm{W}.$
\end{enumerate}
Then $\C$ form a model category with weak equivalences  (resp. fibrations) $\mathrm{W}$  (resp. $\mathrm{F})$. More over it is cofibrantly generated with generating cofibrations $G(\mathrm{I})$ and generating trivial cofibrations $G(\mathrm{J})$.
 \end{lemma}

We prove the main theorem using \ref{lem1}.

\begin{lemma}
The poushout of  $|U\Lambda^{n}_{i}|\rightarrow |U\Delta^{n}|$ by a morphism  $\mathrm{F}:|U\Lambda^{n}_{i}|\rightarrow \D$ is a weak equivalence.
\end{lemma}
\begin{proof}
See \ref {tech2}
\end{proof}
\begin{lemma}
The poushout of  $\{x\}\rightarrow |\mathcal{H}|$ by $\{x\}\rightarrow\C$ is a weak equivalence for all  $\C\in \Cat_{\Top}$.
\end{lemma}
\begin{proof}
Let $\mathcal{O}$ the set of objects of $\C$ without the object $\{x\}$ touched by the morphism $\{x\}\rightarrow \C$. We note by  $x,~y$ objects of $|\mathcal{H}|$.
The goal is to prove that $h$ defined in the following pushout is a weak equivalence 

 $$
  \xymatrix{
   \{x\} \ar[r]\ar[d]  & \C \ar[d]^{h} \\
     |\mathcal{H}| \ar[r]& \D
  }
  $$
Observe that there is an other double pushout

 $$
  \xymatrix{
  \{x\}\sqcup \mathcal{O}\ar[rr]\ar[d] && \C\ar[d]^{i} \\
   \{x,y\}\sqcup \mathcal{O}\ar[rr]\ar[d]  && \C\sqcup\{y\} \ar[d]^{h^{'}} \\
     |\mathcal{H}|\sqcup \mathcal{O} \ar[rr]&& \D.
  }
  $$
  Which is a consequence of: 
  $$ |\mathcal{H}|\sqcup \mathcal{O}\bigsqcup_{\mathcal{O}\sqcup \{x,y\}} \C\sqcup\{y\}= |\mathcal{H}|\bigsqcup_{\{x,y\}} \C\sqcup\{y\}=|\mathcal{H}|\bigsqcup_{\{x\}} \C=\D.$$
 The morphism $h^{'}$ is a natural extension  of $h$, i.e., $h^{'}\circ i= h$.

  On the other hand, the counity $ c: |\sing\C|\rightarrow \C $ is a weak equivalence. Consider the following pushout  in $\Cat_{\sSet}$:
    $$
  \xymatrix{
   \{x\}\sqcup \mathcal{O}\ar[r]\ar[d] & \sing\C \ar[d]^{i} \\
   \{x,y\}\sqcup \mathcal{O}\ar[r]\ar[d]  &  \sing(\C)\sqcup\{y\} \ar[d]_{f^{'}} \\
     \mathcal{H}\sqcup \mathcal{O} \ar[r]& \D^{'}.
  }
  $$
Since $\Cat_{\sSet}$ is a model category, we have that $f=f^{'}\circ i$ is a weak equivalence. Consequently $|f|$ is a weak equivalence in $\Cat_{\Top}$ .\\
As before  $f^{'}$ is an extension of $f$.

    Using the fact that the functor  $|-|$ commutes with colimits, the diagram of the following double pushout  permit to  conclude:
   $$
  \xymatrix{
                                                                   &&   |\sing\C|\ar[rr]^{\sim}\ar[d]^{i} && \C\ar[d]^{i} \\
   \{x,y\}\sqcup \mathcal{O}\ar[rr]\ar[d]  &&  |\sing(\C\sqcup\{y\})| \ar[d]_{|f^{'}|} \ar[rr]^{c} && \C \sqcup\{y\}\ar[d]_{h^{'}}\\
     |\mathcal{H}|\sqcup \mathcal{O} \ar[rr]&& |\D^{'}|\ar[rr]_{m} && \D.
  }
  $$
 In Fact, 
  $$m:\D=(|\mathcal{H}|\sqcup \mathcal{O}) \star|\sing(\C\sqcup\{y\})|\rightarrow ( |\mathcal{H}|\sqcup \mathcal{O}) \star (\C\sqcup\{y\})=\D^{'}$$
is a weak equivalence by \ref{tech5}. We have seen that  $|f|$ is a weak equivalence, so by the property  "2 out of 3" we conclude that $h$ is a weak equivalence.
\end{proof}
\begin{lemma}\label{cattop1}
The functor $\sing$ commutes with directed colimites. 
\end{lemma}
\begin{proof}
Let $\lambda$ be an ordinal and let 
$$\C=\colim_{\lambda}\C_{\lambda} ,$$ 
a directed colimit in $\Cat_{\Top}$. If $a^{'}$ and $b^{'}$ are two objects in $\C$, then by definition, there exists  an index $t$ such that they are repretented 
by  $a, ~b\in \C_{t},$ and $\Map_{\C}(a^{'},b^{'})$ is a colimit of the following diagram:
$$\Map_{\C^{a,b}_{t}}(a,b)\rightarrow \dots \Map_{\C_{s}}(a_s,b_s)\rightarrow \Map_{\C_{s+1}}(a_{s+1},b_{s+1})\rightarrow\dots...$$
where $\C^{a,b}_{t}$ is a full subcategory of $\C_{t}$ with only two objects $a,~b$
Since the functor $\Ob: \Cat\rightarrow \mathbf{Set}$ and the functor  $\sing: \Top\rightarrow\sSet$ commute with  directed colimits, we have that 
$\sing:\Cat_{\Top}\rightarrow\Cat_{\sSet}$ commutes with directed colimites .
\end{proof}
\begin{lemma}
The objects $ |U\Lambda^{n}_{i}|,~ |U\Delta^{n}|$ and $|\mathcal{H}|$ are small in $\Cat_{\Top}$
\end{lemma}
\begin{proof}
It is a consequence of  the fact that $ U\Lambda^{n}_{i},~ U\Delta^{n}$ , $\mathcal{H}$ are small in  $\Cat_{\sSet}$ and 
$\sing:\Cat_{\Top}\rightarrow\Cat_{\sSet}$ commutes with directed colimits.
\end{proof}
\begin{lemma}
The transfinite composition of weak equivalences in  $\Cat_{\sSet}$ is a weak equivalences.
\end{lemma}
\begin{proof}
It is a consequence that the transfinite composition of weak equivalences in  $\sSet$ and $\Cat$ is a weak equivalence.
Note that $\pi_{0}:\Cat_{\sSet}\rightarrow \Cat$ commutes with colimits because it admit a right adjoint:  the Functor which correspond to each topological enriched category  $\C$ an trivially enriched category i.e., we forget the topology of $\C$. 
\end{proof}
\begin{corollary}
The category  $\Cat_{\Top}$ is a cofibrantly generated model category Quillen equivalent to $\Cat_{\sSet}$.
\end{corollary}


\section{Graphs and Categories }
In this paragraph, we define an adjunction between $\Cat_{\Top}$ and the categories of enriched graphs  on $\Top.$ This adjunction is constructed in the particular case where the set of objects is fixed.
We will denote  $\mathcal{O}-\Cat_{\Top}$ the category of small enriched categories over  $\Top$ with fixed set of objects  $\mathcal{O}$, the morphisms are those functors which are identities on objects.
By the same way, we define the category of small graphs enriched over $\Top$ by  $\mathcal{O}-\mathbf{Graph}_{\Top}$ with a fixed set of vertices $\mathcal{O}$ \\
There exists an adjunction between  $\mathcal{O}-\Cat_{\Top}$ and $\mathcal{O}-\mathbf{Graph}_{\Top}$ given by the forgetful functor and the free functor.
Before starting, we define the free functor between graphs and categories. First we study the case where 
 $\mathcal{O}$ is a set with one element. 
\begin{lemma}
There exists a right adjoint to the forgetful functor  $U:\mathbf{Mon}\rightarrow\Top$ where $\mathbf{Mon}$ is the category of topological monoids.
\end{lemma}
\begin{proof}
Let $X$ in $\Top$. we define
$$L(X) = \ast\sqcup X\sqcup (X\times X)\sqcup ( X\times X\times X)\sqcup\dots ;$$
it is a functor from  $\Top$ to topological monoids.

 It is easy to see that $L:\Top\rightarrow\mathbf{Mon}$ is a well defined functor. In fact, it is the desired functor. Let$M$ be a topological monoids,
   a morphism of monoid $L(X)\rightarrow M$ is given by a morphism of non pointed topological spaces  
   $X\rightarrow U(M).$This morphism  extends in an unique way in a morphism of monoids if we consider the following morphisms in $\Top$: 
 $$X\times X\dots \times X\rightarrow M\times M\dots \times M\rightarrow M.$$
We conclude that:
 $$\homs_{\Top}(X,U(M))=\homs_{\mathbf{Mon}}(L(X),M).$$
\end{proof}

For a generalization of the precedent adjunction to an adjunction between $\mathcal{O}-\Cat_{\Top}$ and $\mathcal{O}-\mathbf{Graph}_{\Top}$, we do ass follow: We pose $\mathbf{O}$ the trivial category with set of object  $\mathcal{O}$. for each graph $\mathbf{\Gamma}$ in  $\mathcal{O}-\mathbf{Graph}_{\Top}$ we define the set of the following categories indexed by  a pair of element $a,~b\in \mathcal{O}$ 
\[
 \mathbf{\Gamma}_{a,b}(c,d) = \left\{ \begin{array}{ll} 
 \mathbf{\Gamma}(c,d)  & \hbox{if ~$c=a\neq b=d$}\\ 
L(\mathbf{\Gamma}(c,d)) & \hbox{if ~ $a=c=b=d$}\\ 
\emptyset & \hbox{if~ $c\neq d$ ~and~ $a\neq c\wedge b\neq d$  }\\
\ast=id & \textrm{else}
\end{array} \right. 
\]

Let $\mathbf{\Gamma}$ a graph in $\mathcal{O}-\mathbf{Graph}_{\Top}$. we define the free category induced by the graph as a free product in the category    $\mathcal{O}-\Cat_{\Top}$ of all categories of the form
$\mathbf{\Gamma}_{a,b}$, more precisely 
$$L(\mathbf{\Gamma}) = \star_{(a,b) \in\mathcal{O}\times \mathcal{O}} \mathbf{\Gamma}_{a,b}.$$
By the free product, we mean the following colimit in $\Cat_{\Top}$:
$$ \colim_{(a,b) \in\mathcal{O}\times \mathcal{O}} \mathbf{\Gamma}_{a,b}.$$


\section{Realization}
Let  $\M$ be a simplicial model category (i.e., tensored and cotensored in a suitable way). The category 
$[\Delta^{op}, \M] $ is a model category with Reedy model structure
(cf \cite{goerss1999}) where the weak equivalences are defined degrewise.
\begin{definition} \label{rel1}
The realization functor
$$|-|:[\Delta^{op}, \M]\rightarrow \M$$
is defined as follow:

$$\xymatrix{ \bigsqcup_{\phi:[n]\rightarrow [m]} M_{m}\otimes\Delta^{n}\ar@<2pt>[r]^-{d_{0}}\ar@<-2pt>[r]_-{d_{1}} & \bigsqcup_{[n]} M_{n}\otimes\Delta^{n}\ar[r] & |M_{\bullet}| }$$
où $d_{0}=\phi^{\ast}\otimes id$ and $d_{1}=id\otimes\phi$.
\end{definition}
\begin{lemma}\label{rel11}
SInce $\M$ is a simplicial category, the functor $|-|$ admit a right adjoint:
$$(-)^{\Delta}:\M\rightarrow [\Delta^{op}, \M]: M\mapsto M^{\Delta^{n}}.$$
\end{lemma}
\begin{lemma}\label{rel2}[\cite{goerss1999},VII, proposition 3.6]
Let  $\M$ a simplicial model category and $[\Delta^{op}, \M] $ a Reedy model category, then the realization functor 
$$|-|:[\Delta^{op}, \M]\rightarrow \M$$
is a left Quillen functor. 
\end{lemma}

Now, we specify to $\M=\Top$.  
In this particular case, $[\Delta^{op}, \Top]$ is a monoidal category (the monoidal structure is defined degree wise form the monoidal structure of $\Top$). So, the realization functor  
$|-|:[\Delta^{op}, \Top]\rightarrow \Top$ commutes with the monoidal product (cf \cite{EKMM}, chapitre X, proposition 1.3).\\
\begin{corollary}\label{rel3}
The realization functor  $|-|:[\Delta^{op}, \Top]\rightarrow \Top$ preserve the homotopy equivalences. 
\end{corollary}

In the practice, the lemma \ref{rel2} is difficulte to use.  It is quite-difficult to show that an object in $[\Delta^{op}, \M]$ is 
Reedy cofibrant. In \textrm{l'appendice A} of  \cite{segal1974}, Segal gives us an alternative solution in the particular case of  $[\Delta^{op}, \Top]$.

\begin{lemma}\label{doublerel}
There exist a functor $||-||:[\Delta^{op}, \Top]\rightarrow \Top$, called \textbf{good realization} with the following propreties: 
\begin{enumerate}
\item Let $f_{\bullet}: X_{\bullet}\rightarrow Y_{\bullet}$ a morphism in $[\Delta^{op}, \Top]$ such that if $f_{n}: X_{n}\rightarrow Y_{n}$ 
is a weak equivalence for all  $n\in\mathbb{N}$,
then $||f_{\bullet}||: ||X_{\bullet}||\rightarrow ||Y_{\bullet}||$ is a weak equivalence in $\Top$;
\item There exists a natural transformation $\mathcal{N}:||-||\rightarrow |-|$, with the property that for all \textbf{good simplicial topological space} 
$X_{\bullet}$, the natural morphism:
$$ \mathcal{N}_{X_{\bullet}}:||X_{\bullet}||\rightarrow |X_{\bullet}|$$  is a weak equivalence in $\Top$;
\item The natural morphism $||X_{\bullet}\times Y_{\bullet}||\rightarrow||X_{\bullet}||\times ||Y_{\bullet}||$ is a weak equivalence in  $\Top$.
\end{enumerate}
\end{lemma}
For  the details we refer to \cite{segal1974}.\\
\begin{lemma}\label{doublerel1}
There exists an endofunctor $\tau: [\Delta^{op}, \Top]\rightarrow [\Delta^{op}, \Top]$ and a natural transformation
 $\mathcal{Q}:\tau\rightarrow id$ 
with the following properties:
\begin{enumerate}
\item   $\tau X_{\bullet}$ is a good simplicial topological space for all $X_{\bullet}\in  [\Delta^{op}, \Top]$;
\item The natural morphism $\mathcal{Q}_{n}:\tau_{n}(X_{\bullet})\rightarrow X_{n}$ is a weak equivalence for all $n\in\mathbb{N}$;
\item The natural morphism $||X_{\bullet}||\rightarrow|\tau(X_{\bullet})|$ is a weak equivalence;
\item Finally , we have  $\tau_{0}(X_{\bullet})=X_{0}$.
\end{enumerate}
\end{lemma}
\begin{corollary}\label{doublerel2}
Let  $f_{\bullet}: X_{\bullet}\rightarrow Y_{\bullet}$ a morphism in $[\Delta^{op}, \Top]$, such that $f_{n}$ is a weak equivalence for all  $n$, then
$$|\tau(f_{\bullet})|: |\tau(X_{\bullet})|\rightarrow |\tau(Y_{\bullet})|$$
is a weak equivalence of topological spaces.
\end{corollary}
\begin{proof}
It is a direct consequence from \ref{doublerel} and \ref{doublerel1}.
\end{proof}
We can see the functor $\tau$ as kind of cofibrant replacement. It is useful to know how to describe  the functor $\tau$.
\begin{definition}\label{tau}[\cite{segal1974}, Appendice A]
Let$A_{\bullet}$ a simplicial topological space and $\sigma$ a subset of $\{1,\dots,n\}$. We pose:
\begin{enumerate}
\item $A_{n,i}=s_{i}A_{n}.$
\item $A_{n,\sigma}=\cap_{i\in\sigma}A_{n,i}$.
\item $\tau_{n}(A_{\bullet})$ is a union of all subsets $[0,1]^{\sigma}\times A_{n,\sigma}$ of $[0,1]^{n}\times A_{n}$.
\end{enumerate}
\end{definition}
The morphism $\tau(A_{\bullet})\rightarrow A_{\bullet}$ collapes  $[0,1]^{\sigma}$ and inject  $A_{n,\sigma}$ in$A_{n}$.

\begin{lemma}\label{tauretracte}
The functor $\tau$ sends homotopy equivalences to homotopy equivalences.
\end{lemma}
\begin{proof}

Let  $h:X_{\bullet}\times[0,1]\rightarrow Y_{\bullet}$ be a homotopy between  $t$ and $s$. By definition of  $\tau$, we have 
\begin{eqnarray*}
\tau_{n}(X_{\bullet}\times [0,1]) &=& \bigcup_{\sigma\in\{1,\dots n\}}[0,1]^{\sigma}\times(X_{\bullet}\times [0,1])_{n,\sigma}\\
&=& \bigcup_{\sigma\in\{1,\dots n\}}([0,1]^{\sigma}\times X_{n,\sigma}\times [0,1])\\
&=&(\bigcup_{\sigma\in\{1,\dots n\}}[0,1]^{\sigma}\times X_{n,\sigma})\times [0,1]\\
&=& \tau_{n}(X_{\bullet})\times [0,1].
\end{eqnarray*}
Consequently  $\tau(h):\tau(X_{\bullet})\times [0,1]\rightarrow \tau(Y_{\bullet})$ is a homotopy between  $\tau(t)$ and $\tau(s)$.
\end{proof}
\begin{definition}
a strong section  $f:X\rightarrow Y$ is a continues application  $i: Y\rightarrow X$ such taht $f\circ i=id_{Y}$ 
and such that there exists a homotopy between $i\circ f$ and  $id_{X}$ which fix $Y$.
\end{definition}
\begin{corollary}\label{tauretracte1}
The functor $\tau$ preserve strong sections  .
\end{corollary}
\begin{proof}
It is a consequence of the lemma \ref{tauretracte} and that $\tau$ is a functor so it preserves the identies. 
\end{proof}
\begin{corollary}\label{tauretracte2}
If $X$ is a constant simplicial topological space, then $\mathcal{Q}_{X}:\tau(X)\rightarrow X$ admit a strong section. 
\end{corollary}
\begin{proof}
The section  $i: X\rightarrow \tau(X)$ is induced by the identity on $X$. To show that it is a strong section, it is suffissant  to see that $\tau_{n}(X)=[0,1]^{n}\times X$ by definition.
\end{proof}


\section{Pushouts in $\Cat_{\V}$}\label{pushout}
We define and compute some (simple) pushouts in the category of small enriched categories
$\V-\Cat$. In our example $\V$ is the category $\sSet$ or $\Top$. For more details see 
(\cite{lurie}, A.3.2).\\
\begin{definition}
Let $U: \V\rightarrow \Cat_{\V}$ be a functor defined as follow:\\
For each object $S\in\V$, $U(S)$ is the enriched category with two objects $x$ and $y$
such that $\Map_{U(S)}(x,y)=S.$
\end{definition}
Let $f: S\rightarrow T$ be a morphism in $\V$ and $\C$ an enriched category on $\V$.
We want to describe explicitly the following pushout diagram:
 
 $$\xymatrix{
  US\ar[r]^{h}\ar[d]^{Uf} &  \C\ar[d]   \\
    UT\ar[r]&  \D
  }
  $$

It is enough claire that the objects of $\C$ and $\D$ are the same. The difficult par is
to define $\Map_{\D}$.\\
Let $w,~z \in \C$ and define the following sequence of objets in $\V$:
\begin{eqnarray*} 
M^{0}_{\C}& = & \Map_{\C}(w,z). \\ 
M^{1}_{\C} & = & \Map_{\C}(y,z)\times T\times\Map_{\C}(w,x). \\ 
M^{2}_{\C} & = & \Map_{\C}(y,z)\times T\times\Map_{\C}(y,x) \times T\times\Map_{\C}(w,x).\\
\dots
\end{eqnarray*} 
More generally, an object of $M_{\C}^{k}$ is given by a finite sequence of the form

$$ (\sigma_{0},\tau_{1},\sigma_{1},\tau_{2},\dots,\tau_{k},\sigma_{k})$$
 where 
 $$\sigma_{0}\in \Map_{\C}(y, z ),~\sigma_{k}\in \Map_{\C}(w, x),~ \sigma_{i} \in\Map(y, x)$$
  for  $0< i  < k$, and $\tau_{i}\in T$ for  $0< i\leq k.$ \\

 We define $\Map_{\D}(w,z)$ as a quotient $\bigsqcup_{k} M_{\C}^{k}$ relative to the
  following relations:
  $$(\sigma_{0} , \tau_{1} , \dots , \sigma_{k}) \sim (\sigma_{0} , \tau_{1} , \dots, \tau_{j-1} , \sigma_{j-1} \circ h(\tau_{j} ) \circ \sigma_{j} , \tau_{j+1} , \dots , \sigma_{k}), $$
  when $\tau_{j}$  is an element of $S\subset T.$ 

The category $\D$ is equiped with the following associative composition:
  $$(\sigma_{0} , \tau_{1} , \dots, \sigma_{k}) \circ (\sigma_{0}^{'} ,  \tau_{1}^{'} ,  \dots,,\sigma_{l}^{'} ) = (\sigma_{0} ,  \tau_{1} ,  \dots, ,  \tau_{k} , \sigma_{k}\circ \sigma_{0}^{'} , \tau_{1}^{1} ,  \dots, \sigma_{l}^{'} ).$$
  Observe that there is a natural filtration on  $\Map_{\D}(w,z)$:
  $$\Map_{\C}(w, z ) = \Map_{D} (w, z )^{0}\subset \Map_{\D} (w, z )^{1} \subset \dots$$
  where  $\Map_{D} (w, z )^{k}$ is defined as image of
  $\bigsqcup_{0\leq i\leq k} M_{\C}^{i}$ in $\Map_{\D}(w,z)$ and 
  $$\bigcup_{k}\Map_{\D}(w,z)^{k}= \Map_{D}(w,z).$$
 The most important fact is that $\Map_{\D}(w,z)^{k}\subset\Map_{\D}(w,z)^{k+1}$ is constructed as  
 pushout of  the inclusion: 
  $  N^{ k+1}_{\C} \subset M^{k+1}_{\C}, $ where $N^{k+1}_{\C}$ is a subobject of $M^{k+1}_{\C}$ of $(2m+1)$-tuples  
  $ (\sigma_{0}, \tau_{1} , . . . , \sigma_{m})$ such that  $\tau_{i}\in S$ for at less one $i$.


\subsection{Monads}
The main goal of this section is to generalize the section 2 de l'article \cite{dwyer1980} to the categories 
enriched over  $\Top$.\\
 Every adjunction define a monad and a comonad. We are interested on the particular adjunction
$$  \xymatrix{ \mathcal{O}-\mathbf{Graph}_{\Top} \ar@<2pt>[r]^{L} & \mathcal{O}-\Cat_{\Top}  \ar@<2pt>[l]^{U} }$$
We have a monad $T=UL$ and a comonad $F=LU$. The multiplication on  $T$ is denoted by $\mu: TT\rightarrow T$ and the unity $\eta: id\rightarrow T$, 
the comultiplication by $\psi:F\rightarrow FF$ and finally the counity by $\phi:F\rightarrow id$. The
$T-$algebras are exactly those graphs which have a structure of a category (composition).

\begin{notation}
The category of small categories enriched over $\Top$ and with fixed set of objects $\mathcal{O}$ is
noted by $\mathcal{O}-\Cat_{\Top}$.\\
We note by $\mathcal{O}-\sCat_{\Top}$ the category of presheaves $[\Delta^{op},\mathcal{O}-\Cat_{\Top}]$ and\\
  $\mathcal{O}-\mathbf{sGraph}_{\Top}$ the category of prescheaves $[\Delta^{op},\mathcal{O}-\mathbf{Graph}_{\Top}]$.\\
If we note $[\Delta^{op},\Top]$ by $\mathbf{sTop}$ then we have
$$\mathcal{O}-\sCat_{\Top}=\mathcal{O}-\Cat_{\mathbf{sTop}},$$ 
and 
$$\mathcal{O}-\mathbf{sGraph}_{\Top}=\mathcal{O}-\mathbf{Graph}_{\mathbf{sTop}}.$$ 
\end{notation}


\subsubsection{Simplicial resolution }

Let $\C$ be an object of  $\mathcal{O}-\Cat_{\Top}$, 

We define the iterated composition of $F$ by :
$$F^{k}=\underbrace{ F\circ F\cdots\circ F }_{k}.$$
The comonad $F$ gives us a simplicial resolution $\C$ (cf \cite{dwyer1980}) defined as follow:
$$F_{k}\C=F^{k+1}\C,$$
With faces and degeneracies:
$$\xymatrix{
   F_{k}\C\ar[rr]^-{d_{i}=F^{i}\phi F^{k-i}} & & F_{k-1}\C  \\
       F_{k}\C\ar[rr]^-{s_{i}=F^{i}\psi F^{k-i}} & & F_{k+1}\C
  }$$
   
The category of compactly generated spaces $\Top$ is a simplicial model category (tensored and cotensored over 
$\sSet$) So we have :
 \begin{enumerate}
 \item In $\mathcal{O}-\sCat_{\Top}$ we have the morphism $f:F_{\bullet }\C\rightarrow \C$, where $\C$ is sow as a constant object in 
 $\mathcal{O}-\sCat_{\Top}$ and t $f_{k}=\phi^{k+1}$.
 \item The morphism $f$ admit a section $i:\C\rightarrow F_{\bullet}\C$ in the category  $\mathbf{Graph}_{\mathbf{sTop}}$.
 The section  $i$ is induced by the unity of the monad $T$ i.e., $\eta_{U\C}: U\C\rightarrow ULU \C$; 
 \item The adjunction 
$$\xymatrix{
[\Delta^{op}, \Top]\ar@<1ex>[r]^-{|-|} & \Top, \ar@<1ex>[l]^-{(-)^{\Delta}}
}$$ 
induce the following adjunction
$$\xymatrix{
\mathcal{O}-\Cat_{\mathbf{sTop}}\ar@<1ex>[r]^-{|-|} &\mathcal{O}-\Cat_{\Top}, \ar@<1ex>[l]^-{(-)^{\Delta}}
}$$ 
since the realization functor is monoidal.
 \item The realization of the morphism $f$ in  $\mathcal{O}-\sCat_{\Top}$ induce a weak equivalence i.e., $ |f|:\Map_{|F_{\bullet}\C|}(a,b)\rightarrow \Map_{\C}(a,b)$ is a weak equivalence in $\Top$ for all  $a,b \in  \mathcal{O}$.
 \end{enumerate}
\begin{remark}
The realization functor $|-|$ does not  "see" the category structure, but only the graph structure.
\end{remark}

More generally, for all $\C,~\D$ in $\mathcal{O}-\Cat_{\Top} $ the following morphism:
$$ \xymatrix{ F_{\bullet}(\C)\star\D \ar[r] & \C\star\D }$$
 admit a strong  section  $ \C\star\D\rightarrow F_{\bullet}\C\star\D$ in the category $\mathcal{O}-\mathbf{sGraph}_{\Top}$. In fact, The category 
 $\mathcal{O}-\mathbf{Graph}_{\Top}$ is monoidal (nonsymmetric) with monoidal product $\times_{\mathcal{O}}$
  which is a generalization of  (\cite{Mac},II, 7). A topologically enriched category is a monoid with respect to this monoidal product. The free product $\C\star\D$ is 
  constructed in  $\mathcal{O}-\mathbf{Graph}_{\Top}$ as
   $$\mathcal{O}^{c}\sqcup_{\mathcal{O}}\C^{'}\sqcup_{\mathcal{O}} \D^{'}\sqcup_{\mathcal{O}}(\C^{'}\times_{\mathcal{O}}\C^{'})\sqcup_{\mathcal{O}}(\D^{'}\times_{\mathcal{O}}\D^{'})\sqcup_{\mathcal{O}}(\C^{'}\times_{\mathcal{O}}\D^{'})\sqcup_{\mathcal{O}}(\D^{'}\times_{\mathcal{O}}\C^{'})\dots$$  
  where $\C^{'}$ (resp. $\D^{'}$) is a correspondant  graph of $ \C$ (resp. $\D$) without identities and 
   $\mathcal{O}^{c}$ is the trivial category obtained from the set $\mathcal{O}.$
  So  $ \C\star\D\rightarrow F_{\bullet}(\C)\star\D $ is induced by the section  $ i: \C\rightarrow F_{\bullet}\C $ and $id:\D\rightarrow \D$.
  consequently the morphism 
  $$ \Map_{\C\star\D}(a,b)\rightarrow \Map_{|F_{i}(\C)\star\D|}(a,b)= \Map_{|F_{i}(\C)|\star\D}(a,b)$$
  is a weak equivalence in $\Top$ for all objects $a,~b \in  \mathcal{O}$.

\begin{lemma}\label{grapheq}
Let $\C\rightarrow \D$ a weak equivalence  in $ \mathcal{O}-\Cat_{\Top}$ and  let $\mathbf{\Gamma}$ a graph in $ \mathcal{O}-\mathbf{Graph}_{\Top}$, the the induced morphism :
$$ L(\mathbf{\Gamma})\star\C\rightarrow  L(\mathbf{\Gamma})\star\D $$
is a weak equivalence in the category $\mathcal{O}-\Cat_{\Top}$.
\end{lemma}
\begin{proof}
It is enough to prove  that $ \C^{'}=L(\mathbf{\Gamma})_{a,b}\star\C\rightarrow  L(\mathbf{\Gamma})_{a,b}\star\D=\D^{'} $ is an equivalence for all  $(a,b)\in \mathcal{O}\times\mathcal{O}$.
If $a \neq b$, it is a direct consequence of the lemma  \ref{tech2}, where we replace $S$ by $\emptyset$ and  $T$ by  $X$. So $\Map_{\C^{'}}(w,z)=\bigsqcup_{k}M^{k}_{\C}$ 
and respectively  $\Map_{\D^{'}}(w,z)=\bigsqcup_{k}M^{k}_{\D}$.
But $M^{k}_{\C}$ is equivalent to $M^{k}_{\D}$ since $\C$ is equivalent to $\D$. We conclude that $\Map_{\C^{'}}(w,z)$ is equivalent to $\Map_{\D^{'}}(w,z).$

If $a = b$, we note the edges from $a$ to $a$ of the graph $\mathbf{\Gamma}$ by $X$. Then we use the precedent case if we remark that  $\C^{'}=L(\mathbf{\Gamma})_{a,b}\star\C$ is simply the following pushout:
     $$\xymatrix{
  U(\emptyset)\ar[r]^{f}\ar[d] &  \C\ar[d]_{g}   \\
U(X) \ar[r]^{\alpha} & \C^{'}
  }
  $$
The morphsim $f$ send the two objects of  $U(\emptyset)$ to $a\in\C$,
so, by the lemma $\ref{tech2}$ we have that  $ L(\mathbf{\Gamma})_{a,a}\star\C\rightarrow  L(\mathbf{\Gamma})_{a,a}\star\D $
 is a weak equivalence.
consequently $ L(\mathbf{\Gamma})\star\C\rightarrow  L(\mathbf{\Gamma})\star\D $ is a weak equivalence by a possibly transfinite composition of weak equivalences.
\end{proof}


 \begin{lemma}\label{tech1}
  Let $i:~X\rightarrow Y$ an inclusion and a weak equivalence of topological spaces and $i(X)$ colsed in  $Y$  such that there exists a  homotopy 
   $H:Y\times[0,1]\rightarrow Y$ which verify the following conditions:
  \begin{enumerate}
  \item $H(-,0)= id_{Y}$
  \item $H(i(x),t) = i(x)$ for all $x\in X.$
  \item $H(-,1)=s$ with $s\circ i=id_{X}.$
  \end{enumerate} 
  
  then the morphism $g$ of the pushout  :
     $$\xymatrix{
  X\ar[r]^{\psi}\ar[d]^{i} &  Z\ar[d]_{g}   \\
Y \ar[r]^{\alpha}\ar@/^/[u]^{s}&  D
  }
  $$
 is a weak equivalence.
   
  \end{lemma}
  \begin{proof}
  We remind that $D=Y\cup_{X} Z.$  To simplify notation be denote the image of  $y\in Y$ in $D$ by $y$, respectively $z$ for the image of $z\in Z$ in $D$ .

  Since $i$ admit a  retraction, $g$ admit also a retraction noted by $s^{'}$ and induced by $s$. It means that we 
  have an inclusion  of $Z$ in $D$ via $g$ beacause of 
  $s^{'}\circ g=id_{Z}$.  In fact, $s^{'}: D\rightarrow Z$ is defined as follow:
   \begin{enumerate}
  \item $s^{'}(z)=z$ for $z\in Z.$
  \item $s^{'}(y)= s(y)$ for $y\in Y.$
   \end{enumerate}
  This new section  $s^{'}$ is well defined by  $s^{'}(\psi(x))=\psi(x)$ and  $s^{'}(i(x))=i(x)$ but in $D$ we have $i(x)=\psi(x)$ for all $x\in X.$ We resume the situation in the following diagram

     $$\xymatrix{
   X\ar[r]^{\psi} \ar@{^{(}->}[d]^{i}& Z \ar@{^{(}->}[d]^{g}\ar@/^/[rdd]^{id}  \\
     Y \ar[r]^{\alpha} \ar@/^/[u]^{s}\ar@/_/[rrd]^{\psi\circ s}  & Y\cup_{X} Z  \ar[rd]^{s^{'}} \\
     & & Z   
  }$$
   
  We construct the homotopy $H^{'}: D\times [0,1]\rightarrow D$ as follow:
  \begin{enumerate}
  \item $H^{'}(-,0) = id_{D}.$
  \item $H^{'}(z, t) = z $ if $z$ is in  $ Z$ .
  \item $H^{'}(y,t)=  H(y,t)$ for all $y$ in $Y$.
  \end{enumerate}
 
  This homotopy is well defined. In fact, it is enough to prove that the gluing operation is well definde. We have $\psi(x)=i(x)$ in $D$, then $H^{'}(i(x),t)=H(i(x),t)=i(x)$ by definition, 
  on the other hand $H^{'}(\psi(x),t)=\psi(x).$
  Since  $i(X)$ is closed in $Y$, then $i(X)$ is closed in  $D$. We conclude that $H^{'}$ is well defined. More over $H^{'}(y,0)=H(y,0)=y$ and so $H^{'}(-,0)$ is the identity.
  
  By simple computation of $H^{'}(-,1): D\rightarrow D$ we have that  $H^{'}(z,1) = z$ for all $z\in Z$ and $ H^{'}(y,1)= H(y,1)= s(y)$ for all $y\in Y$. So, $H^{'}(-,1)=s{'}$ , 
  That means the morphism $s^{'}:D\rightarrow Z\subset D$ is a weak equivalence
  since it is hompotopic to the identity. Consequently $g$ est aussi is a homotopy equivalence because  $s^{'}\circ g=id.$

   \end{proof}



 \begin{lemma}\label{tech2}
 With the precedent notation of graphs \ref{pushout}, if we pose $f: ~S=|\Lambda^{n}_{i}|\rightarrow T=|\Delta^{n}|$, 
 then,
 $\Map_{\C}(w,z)\subset \Map_{\D}(w,z)$ is a weak equivalence  $\forall w,~z\in \C$. 
 \end{lemma}
 \begin{proof}
 We remind here that $\V=\Top$. Since all objects in  $\Top$ are fibrant, so  $f$ admit a section $s$.
 On the other hand, the inclusion $  N^{ k+1}_{\C} \subset M^{k+1}_{\C} $ is a weak equivalence and admit also a section. We will do the demonstration for the case  $k=2$. We use the following notations:

    \begin{eqnarray} 
 A_{0}=\Map_{\C}(y,z)\times S\times\Map_{\C}(y,x) \times S\times\Map_{\C}(w,x) \\ 
 A_{1}=\Map_{\C}(y,z)\times S\times\Map_{\C}(y,x) \times T\times\Map_{\C}(w,x) \\ 
  A_{2}=\Map_{\C}(y,z)\times T\times\Map_{\C}(y,x) \times S\times\Map_{\C}(w,x) 
\end{eqnarray} 
The evident inclusions  are weak equivalences which admit sections induced by  $s$ 
 $A_{0}\rightarrow A_{i},~i=1,2. $

We define the complement of  $N^{ 2}_{\C}$, which consist on  tuples $(a,s_{1},b,s_{2},c) $ in 
$\Map_{\C}(y,z)\times T\times\Map_{\C}(y,x) \times T\times\Map_{\C}(w,x) $ such that
$s_{1}, ~s_{2}\notin S$. We will do our argument in low dimension $n=1$, the rest is similar.
The space $T\times S\cup_{S \times S}S\times T$ is a gluing  of two  intervals $[0,1]$  at the point 0 and  $T\times T$ is simply $[0,1]\times [0,1].$
If we pose  $f:~X=T\times S\cup_{ S \times S}S\times T \rightarrow T\times T=Y$, we are exactly in the situation of
the lemma \ref{tech1} i.e., 
there exist a homotopy between  $X$ and $Y$ which is identity map on $X$. If we rewrite $N^{ 2}_{\C}$ by
$$N^{ 2}_{\C}=A_{1}\bigcup_{A_{0}}A_{2}= X\times  \Map_{\C}(y,z)\times\Map_{\C}(y,x)\times\Map_{\C}(w,x), $$
and $M^{2}_{\C}$ by 
$$ M^{2}_{\C}= Y\times  \Map_{\C}(y,z)\times\Map_{\C}(y,x)\times\Map_{\C}(w,x),$$
The the induced morphism $ N^{ 2}_{\C}\rightarrow M^{2}_{\C}$ verify the condition of the lemma  \ref{tech1}.
Consequently, the pushout of $  N^{ 2}_{\C} \subset M^{2}_{\C} $ by  $N^{2}_{\C}\rightarrow \Map_{D}(w,z)^{1}$ is also a weak equivalence. 
Which means that the  inclusion $\Map_{D}(w,z)^{1}\subset \Map_{D}(w,z)^{2}$  is a weak equivalence.
By the same argument we prove the statement for all $k$ and use the fact that a transfinite composition of weak equivalences is a weak equivalence.  So $$\Map_{\C}(w,z)\dots\subset\Map_{\D}(w,z)^{k}\subset\Map_{\D}(w,z)^{k+1}\dots \subset \Map_{\D}(w,z)$$
 is a weak equivalence. 

 \end{proof}
 
 \begin{corollary}\label{grapheq2}
Let  $\M$ in $ \mathcal{O}-\Cat_{\Top}$, then $F_{i}\M\star\C\rightarrow F_{i}\M\star\D$ is a weak equivalence in $ \mathcal{O}-\Cat_{\Top}$ for all $0\leq i.$
\end{corollary}
\begin{proof}
It is enough to see that $F=LU$ and applied the lemma \ref{grapheq} by putting $\mathbf{\Gamma}=U\M$.
\end{proof}
  
\begin{lemma}\label{tech5}
Let $\C,~\D$ and $\M$ in $ \mathcal{O}-\Cat_{\Top}$ , and $\C\rightarrow\D$ a weak equivalence. Then
$$\M\star\C\rightarrow\M\star\D$$
is a weak equivalence.
\end{lemma}
\begin{proof}
We have seen by \ref{grapheq2} that
$$h_{i}:F_{i}(\M)\star\C\rightarrow F_{i}(\M)\star\D$$ 
is a weak equivalence for all  $0\leq i$.
Consider the following commutative diagram in $\mathcal{O}-\mathbf{Graph}_{\mathbf{sTop}}$:
$$\xymatrix {
     \tau(F_{\bullet}(\M)\star\C)\ar[rr]^{\tau(h_{\bullet})} \ar[dd]_{\tau(t)} \ar[dr]^{f_{\bullet}} && \tau(F_{\bullet}(\M)\star\D) \ar[dr]^{g_{\bullet}} \ar[dd]^>>>>{\tau(s)}  |{}\hole \\
    & F_{\bullet}(\M)\star\C \ar[rr]^{h_{\bullet}~~~~~~~~~} \ar[dd]_>>>>{t} && F_{\bullet}(\M)\star\D \ar[dd]^{s} \\
    \tau(\M\star\C) \ar[rr]^{~~~~~~~\tau(h)} |{}\hole \ar[dr]^{f} && \tau(\M\star\D) \ar[rd]^{g} \\
    & \M\star\C \ar[rr]^{h} &&  \M\star\D \\
  }$$

The morphism $t$ and $s$ are homotopy equivalences. By \ref{rel3}, the morphisms  $|t|$ and  $|s|$ are also homotopy equivalences (of underling graphs).\\
The morphisms  $\tau(t)$ and $\tau(s)$ are homotopy equivalences by \ref{tauretracte1}. And by \ref{rel3}, the morphisms  $|\tau(t)|$ and  $|\tau(s)|$ are homotopy equivalences.\\
The morphism  $|\tau(h_{\bullet})|$  is a weak equivalence \ref{doublerel2}.\\
By the property "2 out of 3" $|\tau (h)|$ is a weak equivalence.\\
 The morphisms  $f$ and $g$ are homotopy equivalences by \ref{tauretracte2}. So $|f|$ and  $|g|$ are also homotopy equivalences by  \ref{rel3}.\\
 We conclude by the property  "2 out of 3" that  $|h|$ is a weak equivalence and so  $h$ is a weak equivalence.
 
 \end{proof}
\section{ $\infty$-categories (quasi-categories)}
In the mathematical literature, there are many models for  $\infty$-categories, for example the enriched categories on Kan complexes \cite{bergner},
 The categories enriched over $\Top$ as we saw before, and the the quasi-categories defined by Joyal. More precisely Joyal constructed a new model structure on  $\sSet$, voir \cite{Joyalcrm}, where the fibrant object are by definition quasi-categories ($\infty$-categories).
We introduce the notion of  \textbf{quasi-groupoïde} 
  which generalize the notion of groupoids in the classical setting of categories. We remind olso the definition of 
  \textbf{coherent nerve} for the enriched categories on $\sSet$ and $\Top$.

\begin{definition}\label{q-cat}
Une $\mathbf{quasi-category}$ is a simplicial set $X$ which has a lifting property for all $0<i<n$:
\begin{equation}\label{q-cat1}
\xymatrix{
    \Lambda^{n}_{i}\ar[r]^-{\forall}\ar[d] & X \ar[d] \\
    \Delta^{n} \ar[r] \ar@{.>}[ru]^-{\exists}& \ast
  }
  \end{equation}
\end{definition}

It is important to remark that the condition $0<i<n$ codify the law composition up to homotopy. Sometimes, we will call such simplicial complexes by weak Kan complexes. For example, if $\C$ is a classical category, then the nerve $\N\C$ is a quasi-category with an additional property: The lifting, is in fact, unique (cf \cite{lurie}, proposition 1.1.2.2). More over a simplicial set is isomorphic to the nerve of a category
 $\C$ if and only if  the lifting \ref{q-cat1} exists and it is unique. 
\begin{lemma}\label{groupoide1}
A category $\C$ is a groupoid iff $\N\C$ is a Kan complex. 
\end{lemma}
\begin{proof}
If $\C$ is a groupoid, then $\N\C$ admit a lifting with respect to $\Lambda^{n}_{n}\rightarrow\Delta^{n}$ and $\Lambda^{n}_{0}\rightarrow\Delta^{n}$ 
simply because all arrows in $\C$ are invertible. So
$\N\C$ is a Kan complex.\\
If $\N\C$  is a Kan complex, we have a lifting with respect to $\Lambda^{2}_{2}\rightarrow\Delta^{2}$. 
That means, every diagram in $\C$
$$\xymatrix{
    & x\ar[d]^{id}  \\
    y \ar[r]^{g} \ar@{.>}[ur]^{f} & x
  }
$$
can be completed by a unique arrow $f:y\rightarrow x$, so $g$ so $g$ is right invertible. We show that $g$ is left invertible using the lifting property with respect to $\Lambda^{2}_{0}\rightarrow\Delta^{2}$. So $\C$ is a groupoid.
\end{proof}
The precedent lemma suggest us a definition for an $\infty$-groupoïde.
\begin{definition}
An $\infty$-category (quasi-category) $X$ is an $\infty$-groupoid (quasi-groupoid)  if it is a Kan complex.
\end{definition}
\begin{example}
Let $Y$ be a topological space, the simplicial set  $\sing Y$ is a Kan complex. so we can see every topological space as an $\infty$-groupoid.
\end{example}

\begin{theorem} \cite{Joyalcrm} (section 6.3)
The category $\sSet$ admit a model structure where the cofibrations are the monomorphisms, the fibrant objects are the quasi-categories, the fibrations are the pseudo-fibrations and the weak equivalences are the categorical equivalences. This is a cartesian cosed model structure. This new structure is noted  by  $(\sSet, \mathbf{Q}).$
\end{theorem}
We don't know if the this new model structure is cofibrantly generated! We will explain later what we mean by categorical equivalences, but we don't describe explicitly the pseudo-fibration.
For each quasi-category $X$ (fibrant object in $(\sSet, \mathrm{Q})$), we can associate its homotopy category (in a classical sens)  noted $ \Ho X$. 
This theory was developed by Joyal, see for example \cite{Joyalcrm}.

\section{Some Quillen adjunctions }
In this paragraph, we describe different Quillen adjunction between $\sSet-\Cat$, $(\sSet,\mathbf{Q})$ and $(\sSet,\mathbf{K})$. 
\subsection{$\sSet-\Cat$ vs $(\sSet,\mathbf{Q})$}
The first adjunction is described  in details in \cite{lurie}. We start by some analogies between classical categories ann simplicial sets.
$$\xymatrix{
\sSet\ar@<1ex>[r]^-{\tau} & \Cat, \ar@<1ex>[l]^-{\N}
}$$ 
The right adjoint is the nerve and the left adjoint associate to each simplicial set its fundamental category. Note that 
this adjunction is not a Quillen adjunction for the two known model structure on $\Cat$  ( Thomason  structure and Joyal structure). We remind the  nerve functor is fully faithful and $ \tau~\N=id$.
The basic idea is to "extend" this adjunction to an adjunction between $(\sSet, \mathbf{Q})$ and the category $\Cat_{\sSet}$. If we use the standard nerve for the enriched categories on simplicial sets, by remembering only the  0-simplices, the we loose all the higher homotopical information. Because of that, we use an other strategy. First we define a left adjoint as follow
$$\Xi : (\sSet,\mathbf{Q})\rightarrow \sSet-\Cat$$
On $\Delta^{n}$, then we apply the left Kan extension.
\begin{definition}\cite{lurie} (1.1.5.1)
The enriched category $\Xi (\Delta^{n})$  has as objects the 0-simplices of $\Delta^{n}$, and
\[
\Xi (\Delta^{n})(i,j) = \left\{ \begin{array}{ll} 
 \N P_{i,j} & \hbox{if ~ $i\leq j$}\\ 
\emptyset  & \hbox{if~ $i>j$} 
\end{array} \right. 
\]

Where $P_{i,j}$ is the set partially ordered by inclusion:
$$\{I \subseteq J : (i, j \in I ) \land (\forall k \in I )[i\leq k \leq j ]\}.$$
\end{definition} 

\begin{definition}
The right adjoint to the functor  $\Xi$  is called the coherent nerve and  noted by $\widetilde{\N}$. It is defined by the following formula:
$$\widetilde{\mathrm{N}_{n}}\C=\homs_{\sSet}(\Delta^{n}, \widetilde{\N} \C):=\homs_{\sSet-\Cat}(\Xi(\Delta^{n}),\C).$$
\end{definition}
Now, we can define the categorical equivalences used in the model stucture $(\sSet, \mathbf{Q})$. We call a morphism of simplicial sets
$f: ~X\rightarrow Y$ une \textrm{categorical equivalence} if $\Xi(f):~\Xi(X)\rightarrow \Xi(Y)$ is an equivalence of enriched categories, i.e., if $\Map_{\Xi(X)}(a,b)\rightarrow \Map_{\Xi(Y)}(\Xi(f) a,\Xi(f) b)$ is a weak equivalence of simplicial sets for all  $a,~b$ and
$\pi_{0}\Xi(f):~\pi_{0}\Xi(X)\rightarrow \pi_{0}\Xi(Y)$ is a equivalence of classical categories.
\begin{theorem}\label{equivalence}
The following adjunction is a  Quillen equivalence  between the Joyal model structure $(\sSet, \mathbf{Q})$ \cite{Joyalcrm}, and the model category on $\Cat_{\sSet}$ defined in \cite{bergner}
$$\xymatrix{
\sSet\ar@<1ex>[r]^-{\Xi} & \sSet-\Cat .\ar@<1ex>[l]^-{\widetilde{\N}}
}$$ 
\end{theorem}
For the proof we refer to \cite{lurie} theorem 2.2.5.1.
\begin{corollary}
Let $\C$ an enriched category on Kan complexes, then the counity
$$\Xi\widetilde{\N}\C\rightarrow \C$$
is a weak equivalence of enriched categories.
\end{corollary}
\subsection{$(\sSet,\mathbf{Q})$ vs $(\sSet,\mathbf{K})$}
In this paragraph, we describe the Quillen adjunction  Between Joyal model structure on simplicial sets  and the classical model structure on $\sSet$ which we note by $(\sSet,\mathbf{K})$, $\mathbf{K}$ for Kan complexes. 
\begin{definition}
The functor $k:\Delta\rightarrow\sSet$ is defined by $k[n]=\widetilde{\Delta^{n}}$ for all $n\geq$, where  $\widetilde{\Delta^{n}}$ is the nerve of the free groupoid generated by
the category $[n]$.
If $X$ is a simplicial set, we define the functor $k^{!}:\sSet\rightarrow\sSet$ by :
$$ k^{!}(X)_{n}=\homs_{\sSet}(\widetilde{\Delta^{n}},X).$$

\end{definition}

The functor $k^{!}$ has a left adjoint $k_{!}$ which is the left Kan extension of  $k$. From the inclusion $\Delta^{n}\subset\widetilde{\Delta^{n}}$
 we obtain, for all $n$, a set morphism
$k^{!}(X)_{n}\rightarrow X_{n}$ which is  $n$-level of a simplicial morphism $\beta_{X}:~k^{!}(X)\rightarrow X$. More precisely, $\beta: ~ k^{!}\rightarrow id$ 
is a natural transformation. Dually, we define a natural transformation $\alpha:~id\rightarrow k_{!}$
\begin{theorem}
The adjoint functors
$$\xymatrix{
(\sSet,\mathbf{Kan})\ar@<1ex>[r]^-{k_{!}} & (\sSet,\mathbf{Q}) .\ar@<1ex>[l]^-{k^{!}}
}$$ 
is a Quillen adjunction. More over, $\alpha_{X}:~X\rightarrow k_{!}(X)$ is an equivalence for each $X.$ 
\end{theorem}
\begin{proof}
For the proof, see (\cite{Joyalcrm}, 6.22).
\end{proof}

\subsection{  $\infty$-groupoids}
In this paragraph, we define a notion of groupoid for categories enriched on simplicial sets or topological spaces, Which we compare with the notion of $\infty$-groupoid defined for quasi-categories.
\begin{definition}
An enriched category $\C$  on $\sSet$ (or $\Top$) is an $\infty$-groupoïde if  $\pi_{0}\C$ is a groupoid in the classical sense of categories.
If $\C$ is enriched on  $\sSet$ ($\Top$),  the $\infty$-groupoid $\C{'}$ associated  to $\C$ is a fibred product in $\Cat_{\sSet}$ (or $\Cat_{\Top}$):
  $$\xymatrix{
   \C^{'}=\iso\pi_{0}\C\times_{\pi_{0}\C} \C \ar[r]\ar[d]  & \C \ar[d] \\
    \iso\pi_{0}\C \ar[r]& \pi_{0}\C.
  }
  $$
\end{definition}
We remark that the functor $\pi_{0}:~\Cat_{\sSet}\rightarrow\Cat $ is a left adjoint, so it does not commute with limits 
in general. But the evident projection $pr: \pi_{0}\C^{'}\rightarrow\iso \pi_{0}\C$ is an isomorphism. In fact, if $w_{1}$ and $w_{2}$ are weak equivalences in $\Map_{\C}(a,b)$ and  $h$ is a homotopy between them 
(i.e. un 1-simplex  in $\Map_{\C}(a,b)$  such that the borders are $w_{1},~w_{2}$)  Then $h$ is also a homotopy in$\Map_{\C^{'}}(a,b)$. This prove that the projection $pr$ is fully faithful. the essential surjectivity  of $pr$ est evident.\\
We note by $G$ the functor which associate  to $\C$ its  $\infty$-groupoid $\C^{'}$. The full subcategory of  $\Cat _{\sSet}$ of  $\infty$-groupoids is noted by $\mathbf{Grp}_{\sSet}$. 
\begin{lemma}\label{groupoide13}
The functor $G: \sSet-\Cat\rightarrow\sSet-\mathbf{Grp}$ is the right adjoint of the inclusion, i.e.,
$$\homs_{\mathbf{Grp}_{\sSet}}(\C,G\D)=\homs_{\Cat_{\sSet}}(\C,\D)$$
$\forall \C\in \mathbf{Grp}_{\sSet}$ and $\D\in\Cat_{\sSet}$.
\end{lemma}
\begin{remark}
We can do the same thing for $\Cat_{\Top}$.
\end{remark}

\begin{proof}
Let  $\C$ be an $\infty$-groupoid and let $\D\in\sSet-\Cat$.  A morphism $f:~\C\rightarrow\D$ define in a unique way an adjoint morphism $g:~\C\rightarrow G\D$ given by the universal map
  $$\xymatrix{
    \C \ar@/^/[rrd] ^{f}\ar@/_/[rdd]^{\phi} \ar@{.>}[rd]^-{\exists !g} \ar[dd]^{q}\\
  & G\D \ar[r] \ar[d] & \D \ar[d] \\
    \pi_{0}\C\ar[r]^-{\pi_{0}f}&  \iso~\pi_{0}\D  \ar[r] & \pi_{0}\D 
  }$$
  
 The morphism $\phi=\pi_{0}f\circ q$ exists and make the diagram commuting, since $\C$ is an   $\infty$-groupoid.
 
\end{proof}

Let $[n]^{'}$ denote the groupoid freely generated by the category $[n]$. An example of $\infty$-groupoid is the category $\Xi k_{!}\Delta^{n}$. In fact, 
$\Xi k_{!} \Delta^{n}= \Xi \N [n]^{'}\rightarrow [n]^{'}$ is a weak categorical equivalence and 
$[n]^{'}$ is fibrant. Since  $[n]^{'}$ is a groupoid groupoid, then $\pi_{0}\Xi k_{!} \Delta^{n}$ is also a groupoid .
\begin{lemma}
Let $\C$ a fibrant category enriched on $\sSet$, then $k^{!}\widetilde{\N}\C=k^{!} \widetilde{\N}\C^{'}$, where $\C^{'}$ is an $\infty$-groupoid associated to $\C$.
\end{lemma}
\begin{proof}
Using the precedent adjunctions, we have for all  $n\geq 0$

\begin{eqnarray} 
(k^{!}\widetilde{\N}\C)_{n} & = &  \homs_{\sSet}(\Delta^{n},k^{!}\widetilde{\N}\C)\\ 
 & = & \homs_{\sSet}(k_{!}\Delta^{n},\widetilde{\N}\C)\\ 
& = &  \homs_{\sSet-\Cat}(\Xi k_{!}\Delta^{n},\C)
\end{eqnarray} 

But $\Xi k_{!}\Delta^{n}$ is an $\infty$-groupoid, so
\begin{eqnarray} 
\homs_{\sSet-\Cat}(\Xi k_{!}\Delta^{n},\C)& = & \homs_{\sSet-\mathbf{Grp}}(\Xi k_{!}\Delta^{n},\C^{'}) \\ 
 & = & \homs_{\sSet-\Cat}(\Xi k_{!}\Delta^{n},\C^{'})\\ 
& = &  \homs_{\sSet}(\Delta^{n}, k^{!}\widetilde{N}\C^{'})\\
&=& (k^{!}\widetilde{\N}\C^{'})_{n}
\end{eqnarray} 
we conclude that  $ k^{!}\widetilde{\N}\C^{'}=k^{!}\widetilde{\N}\C$.
\end{proof}

\begin{definition}\cite{bergner}
In Bergner's model structure on $\Cat_{\sSet}$ \cite{bergner} a morphism $F:\C\rightarrow\D$ is a fibration if
\begin{enumerate}
\item $\Map_{\C}(a,b)\rightarrow\Map_{\D}(Fa,Fb)$ is a fibration of simplicial sets for all $a,~b\in\C$.
\item $F$ has a lifting property of weak equivalences, i.e. it is  Grothendieck fibration for weak equivalences. 
\end{enumerate}
\end{definition}

\begin{corollary}\label{corinfini}
Let $\C^{'}$ the  $\infty$-groupoid associated to the enriched category $\C$ over Kan complexes (or $\Top$), then
$$\widetilde{\N} \C^{'}\rightarrow \N\iso~\pi_{0}\C$$
 pseudo-fibration (cf. \cite{Joyalcrm} ) in $(\sSet,\mathbf{Q})$. 
\end{corollary}

\begin{proof}
Remark that if $\C$ is  fibrant, then $\C\rightarrow\pi_{0}\C$ is a fibration. The Bergner's model structure is right proper so  $\C^{'}\rightarrow \iso~\pi_{0}\C$ is also a fibration. More over, the groupoid  $\iso~\pi_{0}\C$ is fibrant,
 and so $\C^{'}$ is.
Consequently  $\widetilde{\N} \C^{'}\rightarrow \widetilde{\N}\iso~\pi_{0}\C$ is a pseudo-fibration in the category $(\sSet, \mathbf{Q})$, So a pseudo fibration between quasi-categories.

But the category $\pi_{0}\C$ is a  "constant"  simplicial category, so$ \widetilde{\N}\iso~\pi_{0}\C=\N\iso~\pi_{0}\C$. 
We conclude that $\widetilde{\N} \C^{'}\rightarrow \N\iso~\pi_{0}\C$ is a peudo-fibration between quasi-category and a Kan complex, see \ref{groupoide1}.
\end{proof}
Let $X$ a quasi-category, Joyal defined the homotopy category $\Ho(X)$ which is a category in the classical sens. The $0$-simplexes of $X$ form the set of objets of  $\Ho(X)$ and the  $1$-simplexes (modulo the homotopy equivalence) form the morphsims of  $\Ho(X)$. An $1-$simplexe in $X$ is called an weak equivalence
if it is represented in $\Ho(X)$ by an isomorphisme.
\begin{definition}
Let $p: X\rightarrow Y$ a morphism between quasi-categories, and let  $w$ a 1-simplex in $X$, then  $p$ is called conservative if:
$$ p(w)~ \textrm{a weak equivalence in Y } \Rightarrow ~ w ~\textrm{a weak equivalence in X} .$$
\end{definition}
\begin{lemma}(\cite{Joyalcrm}, 4.30)\label{pseudofib}
Let $p: X\rightarrow Y$ a morphism between quasi-categories, such that $p$ is a  pseudo-fibration and conservative. If $Y$ is a Kan complex, then $X$ is.
\end{lemma}
\begin{lemma}\label{inftykan}
Let $\C\in \Cat_{\sSet}$ fibrant, then $\widetilde{N}\C^{'}$ is a Kan complex, where $\C^{'}$ is the $\infty$-groupoid associated  to $\C$.
\end{lemma}
\begin{proof}
We have seen by the corollary \ref{corinfini} that if $\C$ is fibrant, then  $\widetilde{\N} \C^{'}\rightarrow \N\iso~\pi_{0}\C$ is a pseudo-fibration between quasi-categories,
 and $\N\iso~\pi_{0}\C$ is a Kan complex.
We must verify that the morphism is conservative, which is an evident fact because all 0-simplices of  $\Map_{\C^{'}}(a,b)$ are weak equivalences by definition.
By the lemma \ref{pseudofib}, we conclude that  $\widetilde{\N} \C^{'}$ is a Kan complex.
\end{proof}
In \cite{Joyalcrm} (Theorem 4.19) , Joyal construct an adjunction between Kan complexes and quasi-categories.
If we note by $\mathbf{Kan}$
 The full subcategory of $\sSet$ of Kan complexes, and by 
$\mathbf{QCat}$  The full subcategory of $\sSet$ of quasi-categories, then the inclusion $\mathbf{Kan}\subset\mathbf{QCat}$ admit a right adjoint noted by $\mathrm{J}$. 
The functor can be interpreted as follow:
for each quasi-category $X$, $\mathrm{J}(X)$ is the quasi-groupoid associated to $X$, and if $X$ is a Kan complex, then $\mathrm{J}(X)=X$.
\begin{lemma}
Let  $X$ a quasi-category (a fibrant object) in $(\sSet,\mathbf{Q})$. The natural transformation $\beta_{X}: k^{!}(X)\rightarrow X$ is factored by
 $ \beta_{X}:k^{!}(X)\rightarrow \mathrm{J}(X)\subset X.$
More over,  $ \beta_{X}:k^{!}(X)\rightarrow \mathrm{J}(X)$ is a trivial Kan fibration.
\end{lemma}
\begin{proof}
See \cite{Joyalcrm}, proposition 6.26.
\end{proof}
\begin{corollary}\label{groupoid2}
Let a fibrante category $\C\in\Cat_{\sSet}$, and $G \C$ the associated $\infty$-groupoid. 
Then $k^{!}\widetilde{\N}(\C)\rightarrow \widetilde{\N}(G\C)$ is a trivial Kan fibration.
\end{corollary}
\begin{proof}
Since $\C$ is fibrant, we have seen that $k^{!}\widetilde{\N}(\C)= k^{!}\widetilde{\N}(G\C)$, and by the precedent  lemma  $ k^{!}\widetilde{\N}(G\C)\rightarrow \mathrm{J}(\widetilde{\N}(G\C)) $ is a trivial Kan fibration.
But $\widetilde{\N}(G\C)$ is a Kan complex, since  $G\C$ is a fibrant  $\infty$-groupoid, so $\mathrm{J}(\widetilde{\N}(G\C))=\widetilde{\N}(G\C).$
\end{proof}
Now, we can see the analogy between $\N\iso$ in the case of calssical categories and  the funclor $k^{!}\widetilde{\N}$ in the case of enriched categories over $\sSet$. In fact, if $\C$ is a classical category, then the functor  $\iso$ sends $\C$ to its associated groupoid  $G\C$ and so $\N\iso \C= \N G\C.$ If $\C$ is a category enriched over Kan complexes,( i.e., $\C$ is fibrant in Bergner's model structure),  then the simplicial set $ k^{!}\widetilde{\N}\C$
is equivalent to $ \widetilde{\N} G\C$ by the corollary  \ref{groupoid2}.\\

 
 \section{mapping space}
 The goal of this section is to describe the mapping space of the model category $\Cat_{\Top}$. Before making progress in this direction, we need some introduction to different model on $\sSet$.
 \begin{notation}
 We will note the category of simplicial sets with Kan model structure by $(\sSet,\mathbf{K})$.
 The Joyal model structure of quasi-categories will be noted by  $(\sSet,\mathbf{Q})$.
 \end{notation} 
 
  \begin{theorem}\label{adjmap}[\cite{DH2010}, theorem 2.12.]
   Let a Quillen adjunction of Quillen model categories :

   $$\xymatrix{
\C \ar@<1ex>[r]^-{G} & \D. \ar@<1ex>[l]^-{F} 
}$$
The there is a natural isomorphisme
$$ \map_{\C}(a,\mathrm{R}Fb)\rightarrow\map_{\D}(\mathrm{L}Ga,b)$$
in  $\Ho(\sSet)$
\end{theorem}
 
 \subsection{Mapping space in $\Cat_{\Top}$ and $\Cat_{\sSet}$}
In this paragraph, we compute $\map$ for the model categories $\Cat_{\sSet}$ and $\Cat_{\Top}$. \\
Suppose that $\C$ is a small enriched category on $\Top$. We define the coherent nerve of  $\C$ by $\widetilde{\N}\sing\C$, and we define the corresponding $\infty$- groupoid $\C^{'}$ by

  $$
  \xymatrix{
   G\C=\iso~\pi_{0}\C\times_{\pi_{0}\C} \C \ar[r]\ar[d]  & \C \ar[d] \\
    \iso~\pi_{0}\C \ar[r]& \pi_{0}\C
  }
  $$
  By applying the functor $\sing$ to this diagram, we obtain also a pullbak diagram since $\sing$ since it is a right adjoint. We note that
  $ \sing~ \pi_{0}\C=\pi_{0}\sing~ \C=\pi_{0}\C$ and $\sing~ \iso \pi_{0}\C=\iso~\pi_{0}\C=\iso~\pi_{0}\sing\C$
    $$\xymatrix{
   G~\sing~\C=\sing(\iso\pi_{0}\C\times_{\pi_{0}\C} \C )\ar[r]\ar[d]  &  \sing~\C \ar[d] \\
    \sing ~\iso~\pi_{0}\C \ar[r]&  \sing~\pi_{0}\C
  }
  $$
We conclude that $$\sing ~G\C = G~\sing~\C.$$
  More over $k^{!}~\widetilde{\N}~\sing ~\C$ is weak equivalent to $\widetilde{\N}~\sing ~G\C$. 
The homotopy type of the mapping space  $\map_{\Cat_{\Top}}(\ast, \C)$ is computed easily using the theorem \ref{adjmap},
and the adjunction
$$\xymatrix{
\sSet \ar@<1ex>[r]^-{\Xi k_{!}} & \Cat_{\sSet}. \ar@<1ex>[l]^-{k^{!}\widetilde{\N}} 
}$$
We conclude that for every  (fibrant) small category enriched on $\sSet$,  we have the following isomorphism in $\Ho(\sSet)$
$$k^{!}\widetilde{\N}\C\sim\map_{\sSet}(\ast, k^{!}\widetilde{\N}\C)\sim\map_{\Cat_{\sSet}}(\ast,\C)$$
and by the same wa, if $\D$ is a small category enriched on $\Top$, then
$$\map_{\Cat_{\Top}}(\ast,\D)\sim k^{!}\widetilde{\N}\sing\D.$$

by the corollary \ref{groupoid2}, we conclude that
$$\map_{\Cat_{\sSet}}(\ast,\C)\sim  \widetilde{\N}G\C.$$
et 
$$\map_{\Cat_{\Top}}(\ast,\D)\sim \widetilde{\N}G\sing\D.$$
In the classical setting of $\Cat$, we know  that $\map_{\Cat}(\A,\B)\sim\N\iso \HOM_{\Cat}(\A,\B)$. If $\A$ is the terminal category $\ast$, then  $\map_{\Cat}(\ast,\B)\sim\N\iso\B$.
More generally, we have that:
$$\map_{\Cat_{\Top}}(|\Xi k_{!}(A)|, \C)\sim \map_{\sSet}(A,k^{!}\widetilde{\N}\sing\C)\sim\Map(A,\widetilde{\N}G\sing\C),$$
where $\Map$ is the right adjoint functor to the cartesian product in $\sSet.$
 Now,the similarity between $\Cat$ and $\Cat_{\sSet}$ is evident.\\

 \section{localization}
 In this paragraph we show how to construct localization for a topological category with respect to a morphsim
 or a set of morphisms. In the classical setting of small categories we know how to define the localization in a functorial way. The idea is quite simple, let $\C\in \Cat$ and $f$ a morphism in $\C$, we want to define a funclor  $\C\rightarrow\mathrm{L}_{f}\C$ and having the following universal property: if $F:\C\rightarrow\D$ is a functor such that $F(f)$ is an isomorpism in $\D$ then there is a unique factorization of $F$ as 
 $$\C\rightarrow\mathrm{L}_{f}\C\rightarrow\D.$$ 
 \begin{notation}
 In this section, the category with two objects $x$ and $y$ and with one non trivial morphism from $x$ to $y$ 
 will be denoted $\A$.\\
 The category with the same objects  $x$ and $y$ and an isomorphism from $x$ to $y$ (resp. 
 from $y$ to $x$ ) will be denoted  $\B$.
 \end{notation}
 
 \begin{lemma}
 The category $\mathrm{L}_{f}\C$ is isomorphic to following pushout in $\Cat$:
  $$
  \xymatrix{
  \A \ar[r]^{f}\ar[d]_{inc}  & \C \ar[d]^{i} \\
     \B \ar[r]& \M
  }
  $$
Where $inc$ is the evident inclusion and $f$ sends the unique arrow in $\A$ to the morphism $f$ in $\C$. 
 \end{lemma}
 \begin{proof}
 Suppose that we have a functor $F:\C\rightarrow \D$ such that the morphism $f$ is sent to an isomorphism. 
 It induce a functor from $\B\rightarrow \D$. By the pushout property we have a unique functor from $\M$ to $\D$
 which factors the functor $F$. So $\mathrm{L}_{f}\C$ is isomorphic to $\M$.
 \end{proof}
 \begin{corollary}
 For any set $S$ of morphism in $\C$ the category $\mathrm{L}_{S}\C$ exist and it is unique up to isomorphism.
 \end{corollary}
 Now, we are interested for the same construction in the enriched setting $\Cat_{\Top}$. The main difference with the classical case is the the existence, we will construct a functorial model for the localization up to homotopy.
 \begin{notation}
 We denote  by $\A^{h}$ the topological category $|\Xi\N\A|$ and by $\B^{h}$ the category $|\Xi\N\B|$
 \end{notation}
 choosing a morphism $f$ in a topological category $\C$ we want to construct a category a category $\mathrm{L}_{f}\C$ with the following property:  given a morphism $F:\C\rightarrow \D$ in $\Cat_{\Top}$ such that 
 $F(f)$ is a weak equivalence in $\D$ then $F$ is factored (unique up to homotopy) as 
 $$\C\rightarrow\mathrm{L}_{f}\C\rightarrow\D$$.
  \begin{lemma}
 The category $\mathrm{L}_{f}\C$ could be taken as  following pushout in $\Cat_{\Top}$:
  $$
  \xymatrix{
  \A^{h} \ar[r]^{f}\ar[d]_{inc}  & \C \ar[d]^{i} \\
     \B^{h} \ar[r]& \M
  }
  $$
  More over $ \pi_{0}\C\rightarrow \mathrm{L}_{\pi_{0}(f)}\pi_{0 }\C$ is a localization in $\Cat$.
  \end{lemma}
  \begin{proof}
  First, we note that the inclusion $inc$ is a cofibration in $\Cat_{\Top}$. 
  The functor $ \A^{h} \rightarrow \C$ is constructed as follow:
  Let $\A\rightarrow\C$ which sends the only nontrivial morphism of $\A$ to $f\in\C$. It induces a map of simplicial 
  sets $\N\A\rightarrow\widetilde{\N}\sing\C$ and by adjunction a functor $|\Xi\N\A|\rightarrow\C$ which is the functor noted $f:\A^{h}\rightarrow\C$ in the diagram. The functor $inc:\A^{h}\rightarrow\B^{h}$ is induced by the functor $inc:\A\rightarrow\B$. Now suppose that we have a functor $\C\rightarrow \D$ which sends $f$ to a weak equivalence in $\D$. The induced functor $\A^{h}\rightarrow \D$ factors by $\A^{h}\rightarrow G\D\rightarrow \D$
  where $G\D$ is the associated groupoid of $\D$ as seen in previews section.\\
  Consider the diagram:
   $$
  \xymatrix{
  \A^{h} \ar[r]\ar[d]^{inc}  & G\D \ar[d]^{i} \\
     \B^{h} \ar[r]& \star
  }
  $$
   and using the adjunctions we have a corresponding diagram in $\sSet$
    $$
  \xymatrix{
  \N\A \ar[r]\ar[d]^{inc'} & \widetilde{\N}\sing G\D \ar[d]^{i'} \\
     \N\B \ar[r]& \star
  }
  $$
  But now $\sing G\D$ is a Kan complex see \ref{inftykan} and $inc'$ is a trivial cofibration in $\sSet$, so there exist a lifting (not unique) $ \N\B\rightarrow\sing G\D$. By adjunction we have a lifting $\B^{h}\rightarrow G\D\rightarrow \D$. So we can define unique morphism (up to homotopy) $\M\rightarrow \D$ and any functor $\C\rightarrow\D$ as before factors (uniquely up to homotopy) by $\C\rightarrow\M\rightarrow \D$. So a functorial model for $\mathrm{L}_{f}\C$ is $\M$ and the localisation map $\C\rightarrow\mathrm{L}_{f}\C$ is a cofibration and in fact an inclusion of enriched categories.
  \end{proof}
 \begin{corollary}
 For any set $S$ of morphism in a topological category $\C$, the topological category $\mathrm{L}_{S}\C$ exist and it is unique up to homotopy. More over the localization map $\C\rightarrow\mathrm{L}_{S}\C$ is a cofibration.
 \end{corollary}
 























\bibliographystyle{plain} 
\bibliography{Cat-Top}

\end{document}